\documentclass[11pt]{article}
\usepackage{amssymb,amsthm,amsmath,algorithm2e,float,cite}

\usepackage{geometry}
\newtheorem{theorem}{Theorem}[section]
\newtheorem{conjecture}[theorem]{Conjecture}

\usepackage{setspace}
\usepackage{lineno}

\begin{document}

\onehalfspace

\title{Approximately locating an invisible agent\\ in a graph with relative distance queries}

\author{Dennis Dayanikli \and Dieter Rautenbach}

\date{}

\maketitle

\begin{center}
Institut f\"{u}r Optimierung und Operations Research, 
Universit\"{a}t Ulm, Ulm, Germany,
\{\texttt{dennis.dayanikli,dieter.rautenbach}\}\texttt{@uni-ulm.de}\\[3mm]
\end{center}

\begin{abstract}
In a pursuit evasion game on a finite, simple, undirected, and connected graph $G$,
a first player visits vertices $m_1,m_2,\ldots$ of $G$,
where $m_{i+1}$ is in the closed neighborhood of $m_i$ for every $i$,
and a second player probes arbitrary vertices $c_1,c_2,\ldots$ of $G$, and learns whether or not 
the distance between $c_{i+1}$ and $m_{i+1}$ is at most 
the distance between $c_i$ and $m_i$.
Up to what distance $d$ can the second player determine the position of the first? For trees of bounded maximum degree and grids, we show that $d$ is bounded by a constant.
We conjecture that $d=O(\log n)$ for every graph $G$ of order $n$, and show that $d=0$ if $m_{i+1}$ may differ from $m_i$ only if $i$ is a multiple of some sufficiently large integer.
\end{abstract}

{\small 
\begin{tabular}{lp{13cm}}
{\bf Keywords:} pursuit and evasion game
\end{tabular}
}

\section{Introduction}

We study a variant of pursuit and evasion games formalized and studied by 
Britnell and Wildon \cite{brwi}, 
Komarov and Winkler \cite{kowi},
Haslegrave \cite{ha},
Seager \cite{se}, and
Rautenbach and Schneider \cite{rs}.
In these games, further studied in \cite{brdierlemo,cachdeerwe,hejoko,s2},
one player tries to catch or locate a second player moving along the edges of a graph,
using information concerning the current position of the second player.
Our game is also played by two players on a 
finite, simple, undirected, and connected graph $G$ known to both of them, 
and proceeds in discrete time steps numbered by positive integers.
One player, called the {\it mouse}, moves along the edges of $G$.
At time $i$, the mouse occupies some vertex $m_i$ of $G$, 
and, if $i$ is at least $2$, then $m_i$ is either $m_{i-1}$ or a neighbor of $m_{i-1}$,
that is, $m_i$ belongs to the closed neighborhood $N_G[m_{i-1}]$ of $m_{i-1}$ in $G$,
and the mouse can be considered to move with unit speed.
The second player, called the {\it cat}, probes vertices of $G$ one by one in the same discrete time steps.
At time $i$, the cat probes some vertex $c_i$ of $G$, 
where $c_i$ can be chosen without any restriction within the vertex set $V(G)$ of $G$.

The essential difference of our game, as compared to those mentioned above, 
consists in the information provided to the cat.
If $i$ is at least $2$, then, after $m_i$ and $c_i$ have been decided by the two players,
the cat learns whether 
\begin{itemize}
\item $d_i\leq d_{i-1}$ or 
\item $d_i>d_{i-1}$,
\end{itemize}
where $d_i$ denotes the distance ${\rm dist}_G(c_i,m_i)$ in $G$ between $c_i$ and $m_i$.
The goal of the cat is to locate the mouse as precisely as possible,
while the goal of the mouse is to hinder being well located. 
To make this more precise, we introduce some further terminology.

A {\it game $g$} on $G$ is a pair of sequences $((m_i)_{i\in \mathbb{N}},(c_i)_{i\in \mathbb{N}})$
of possible moves $m_i$ for the mouse, and $c_i$ for the cat.
For such a game $g$,
and an integer $i$ at least $2$, let 
$$b_i=
\begin{cases}
1 & \mbox{, if $d_i\leq d_{i-1}$, and}\\
0 & \mbox{, if $d_i>d_{i-1}$,}
\end{cases}$$
that is, the information available to the cat for its choice of $c_{i+1}$ consists of $G$ and the $i-1$ {\it bits} $b_2,\ldots,b_i$.
Note that the cat chooses $c_1$ and $c_2$ without any information about the whereabouts of the mouse.
Based on the available information, 
the cat knows that $m_i$ belongs to the set $M_i$,
where $M_i$ is the set of all vertices $u$ of $G$ 
such that there are vertices $\tilde{m}_1,\ldots,\tilde{m}_i$ of $G$ with 
\begin{itemize}
\item $\tilde{m}_i=u$,
\item $\tilde{m}_j\in N_G[\tilde{m}_{j-1}]$ for every $j\in [i]\setminus \{ 1\}$, and 
\item ${\rm dist}_G(c_j,\tilde{m}_j)\leq {\rm dist}_G(c_{j-1},\tilde{m}_{j-1})$ if and only if $b_j=1$ for every $j\in [i]\setminus \{ 1\}$,
\end{itemize}
where $[i]$ denotes the set of positive integers at most $i$.

If the {\it radius} ${\rm rad}_G(M)$ of a set $M$ of vertices of $G$ is defined as 
$$\min\Big\{ \max\Big\{ {\rm dist}_G(u,m):m\in M\Big\}:u\in V(G)\Big\},$$
then the cat wants to minimize the radius of $M_i$.
Note that the vertex $u$ in this definition may not belong to $M$. 

We say that the cat follows a {\it strategy} $(c_1,c_2;f)$
in the game $((m_i)_{i\in \mathbb{N}},(c_i)_{i\in \mathbb{N}})$ on $G$
if $c_1$ and $c_2$ are vertices of $G$, and 
$$f:\bigcup_{i\in \mathbb{N}}\{ 0,1\}^i \to V(G)$$
is a function
such that $c_1$ and $c_2$ are the two first vertices probed by the cat, and 
$c_{i+1}=f(b_2,\ldots,b_i)$ for every integer $i$ at least $2$.
Furthermore, we say that {\it the cat can localize the mouse up to distance $d$ within time $t$ on $G$}
if there is some strategy $\sigma$ 
such that for every game $((m_i)_{i\in \mathbb{N}},(c_i)_{i\in \mathbb{N}})$ on $G$
in which the cat follows the strategy $\sigma$,
there is some positive integer $i$ at most $t$ with ${\rm rad}_G(M_i)\leq d$.

While the cat only knows $G$ and the $b_i$,
and therefore also the set $M_i$,
we may assume that the mouse knows $G$ and also any strategy followed by the cat.
Note that we consider a game to be infinite, 
and that we did not specify any winning conditions. 
A reasonable way to do so 
is to fix a distance threshold $d$, 
and to declare the cat to be the winner on the pair $(G,d)$
if it can locate the mouse up to distance $d$ within finite time.

A natural question concerning our game is how precisely the cat can localize the mouse on a given graph.
Our first result provides an answer for trees of bounded maximum degree.

\begin{theorem}\label{theoremtree}
The cat can localize the mouse up to distance $4\Delta-6$ within time $O(h\Delta)$ on every tree $T$ 
of maximum degree $\Delta$ at least $2$ and radius $h$.
\end{theorem}
Another natural type of graphs to consider are grids, that is, the Cartesian product $P_n\Box P_m$ of paths.
For these we show the following.
\begin{theorem}\label{theoremgrid}
The cat can localize the mouse up to distance $8$ within time $O(\log n)$ on the grid $P_n\Box P_n$.
\end{theorem}
Both our results concern graphs of bounded maximum degree, 
and we pose the following general conjecture.

\begin{conjecture}\label{conjecture}
The cat can localize the mouse up to distance $O(\log n)$ on a connected graph $G$ of order $n$.
\end{conjecture}
The reason for the $O(\log n)$ term in this conjecture
is that this many bits suffice to identify each vertex,
while the mouse may move this many units of distance 
in the time needed to acquire this many bits.

Our final result establishes a weak version of Conjecture \ref{conjecture}.
In order to facilitate the task for the cat, we slow down the mouse as follows.
For some positive integer $k$,
a game $((m_i)_{i\in \mathbb{N}},(c_i)_{i\in \mathbb{N}})$ on a graph $G$
is {\it $k$-slow} 
if $m_i=m_{i-1}$ for every integer $i$ at least $2$ 
such that $(i-1)\not\equiv 0\mod k$,
that is, the mouse can be considered to move with speed $1/k$.
We say that {\it the cat can localize a $k$-slow mouse up to distance $d$ on $G$}
if there is some strategy $\sigma$ 
such that for every $k$-slow game $((m_i)_{i\in \mathbb{N}},(c_i)_{i\in \mathbb{N}})$ on $G$
in which the cat follows the strategy $\sigma$,
there is some positive integer $i$ with ${\rm rad}_G(M_i)\leq d$.

\begin{theorem}\label{theoremslow}
If $\Delta$ is an integer at least $2$,
then the cat can localize a $4\Delta$-slow mouse up to distance $0$ on every connected graph $G$ of maximum degree at most $\Delta$.
\end{theorem}
In Section \ref{sec2} we prove our results,
and in Section \ref{sec3} we present some open problems.

\section{Proofs}\label{sec2}

For all three of our results, 
we give simple proofs 
capturing essential observations.
For Theorems \ref{theoremtree} and \ref{theoremgrid},
minor improvements are possible
at the cost of tedious case analysis.

\begin{proof}[Proof of Theorem \ref{theoremtree}]
Let $T$, $\Delta$, and $h$ be as in the statement.
In order to express information gathered by the cat, 
we introduce the notation $u\stackrel{i}{\to}v$,
where $uv$ is an edge of $T$, and $i\in \mathbb{N}$,
meaning that $m_i$ belongs to the component of $T-uv$ that contains $v$. 

\begin{algorithm}[H]
\LinesNumbered\SetAlgoLined
\KwIn{The relative distance information ``$d_i\leq d_{i-1}$'' or ``$d_i>d_{i-1}$'' after specifying $c_i$ for some integer $i$ at least $2$.}
\KwOut{A statement of the form ``${\rm dist}_G(r,m_i)\leq 4\Delta-6$'' at the end of some round $i$.}
\Begin{
\LinesNumbered\SetAlgoLined
$r\leftarrow r_0$;
$I\leftarrow \emptyset$;
$i\leftarrow 1$\;
\While{$T_r$ has depth more than $4\Delta-6$\label{l3}}
{
Let $X$ be the set of children of $r$\;
\While{$|X|\geq 2$\label{l5}}
{
Let $u$ and $v$ be distinct vertices in $X$;
$c_i\leftarrow u$; 
$c_{i+1}\leftarrow v$\;
\eIf{$d_{i+1}\leq d_i$}
{
$X\leftarrow X\setminus \{ u\}$;
$I\leftarrow I\cup \left\{ u\stackrel{i+1}{\to}r\right\}$\;
}
{
$X\leftarrow X\setminus \{ v\}$;
$I\leftarrow I\cup \left\{ v\stackrel{i+1}{\to}r\right\}$\;
}
$i\leftarrow i+2$\;
}
Let $r^+$ be the unique element in $X$\label{l14}\;
Let $Y$ be the set of children of $r^+$\;
\texttt{replace}$\leftarrow$ \texttt{false}\;
\While{\texttt{replace}=\texttt{false} and $Y\not=\emptyset$\label{l17}}
{
Let $u$ be in $Y$;
$c_i\leftarrow u$;
$c_{i+1}\leftarrow r$\;
\eIf{$d_{i+1}\leq d_i$}
{
$Y\leftarrow Y\setminus \{ u\}$;
$I\leftarrow I\cup \left\{ u\stackrel{i+1}{\to}r^+\right\}$\;
}
{
\texttt{replace}$\leftarrow$ \texttt{true}\label{l22};
$I\leftarrow I\cup \left\{ r\stackrel{i+1}{\to}r^+\right\}$\;
}
$i\leftarrow i+2$\;
}
\eIf{\texttt{replace}=\texttt{true}}
{
$r\leftarrow r^+$\label{l27}\;
}
{
\Return{``${\rm dist}_G(r,m_i)\leq 4\Delta-6$''}; {\bf break}\label{l29}\;
}
}
\Return{``${\rm dist}_G(r,m_i)\leq 4\Delta-6$''}\label{l32}\;
}
\caption{{\sc Cat on a Tree}}\label{alg1}\end{algorithm}

\pagebreak

A key ingredient for the cat's strategy is the following simple observation:
If $u$ and $w$ are two vertices of $T$ with a common neighbor $v$,
and the cat chooses $c_i$ equal to $u$ and $c_{i+1}$ equal to $w$ 
for some $i\in \mathbb{N}$,
then 
$d_{i+1}\leq d_i$ implies $u\stackrel{i+1}{\to}v$, while
$d_{i+1}>d_i$ implies $w\stackrel{i+1}{\to}v$.
In fact, if $m_{i+1}$ does not belong to the component of $T-uv$ that contains $v$, then 
$d_{i+1}={\rm dist}_T(c_{i+1},m_{i+1})={\rm dist}_T(c_i,m_{i+1})+2\geq {\rm dist}_T(c_i,m_i)+1>d_i$,
which proves the first implication,
and the proof of the second implication is similar.

The cat chooses a center vertex $r_0$, and considers $T$ to be rooted in $r_0$. By the choice of $r_0$, the depth of $T$ is $h$.
For a vertex $r$ of $T$, let $T_r$ denote the subtree of $T$ rooted in $r$ that contains $r$ and all descendants of $r$.

The algorithm {\sc Cat on a Tree}, cf. Algorithm \ref{alg1}, specifies the strategy for the cat. 
Note that the game progresses one round 
whenever the cat specifies some $c_i$.
Throughout the game, the cat maintains a local root $r$ initially equal to $r_0$. The set $I$ contains the information of the form ``$u\stackrel{i}{\to}v$'' already gathered by the cat, and it is initially empty.
In each iteration of the outer {\bf while}-loop in line \ref{l3},
the cat either replaces $r$ with one of its children $r^+$ in line \ref{l27} or returns the statement ``${\rm dist}_G(r,m_i)\leq 4\Delta-6$'' in line \ref{l29} and terminates.
At the beginning of each iteration of the outer {\bf while}-loop,
we have 
\begin{eqnarray}\label{e1} 
\mbox{\it either $i=1$ or $\Big($ $i>1$ and $m_{i-1}$ belongs to $T_r$$\Big)$.}
\end{eqnarray}
The second case is expressed by the element 
$r^-\stackrel{i-1}{\to}r$ of $I$
added in line \ref{l22}, where $r^-$ is the parent of $r$.
For the first iteration, (\ref{e1}) is trivial, and for later iterations, 
we will show it by an inductive argument.
The outer {\bf while}-loop is no longer performed
if the depth of $T_r$ is at most $4\Delta-6$,
which, by (\ref{e1}), implies that the output in line \ref{l32} is correct.

Now, we consider an iteration of the outer {\bf while}-loop,
and assume that (\ref{e1}) holds at its beginning.
Let $i_0+1$ be the value of $i$ at the beginning of that iteration,
that is, $i_0=0$ for the very first iteration.
The first inner {\bf while}-loop in line \ref{l5} exploits the key observation made above.
It follows that there is an ordering $u_1,\ldots,u_x$ of the children of $r$ in $T$ such that, at the end of the first inner {\bf while}-loop, 
\begin{eqnarray}\label{e2}
\mbox{\it $u_j\stackrel{i_0+2j}{\to}r\in I$ for every $j$ in $[x-1]$},
\end{eqnarray}
and the vertex $r^+$ in line \ref{l14} equals $u_x$.
Within the second inner {\bf while}-loop in line \ref{l17},
the cat may replace $r$ with $r^+$ in line \ref{l27},
in which case it concludes 
the current iteration of the outer {\bf while}-loop.
Note that this happens exactly if 
\texttt{replace} is set to \texttt{true},
and $r\stackrel{i+1}{\to}r^+$ is added to $I$ in line \ref{l22},
that is, (\ref{e1}) holds at the end of the considered iteration of the outer {\bf while}-loop,
which concludes the inductive proof of (\ref{e1}).

If \texttt{replace} is never set to \texttt{true} during the second inner {\bf while}-loop, then, similarly as for the first inner {\bf while}-loop,
it follows that there is an ordering $v_1,\ldots,v_y$ of the children of $r^+$ in $T$ such that 
at the end of the second inner {\bf while}-loop, 
the set $I$ contains 
\begin{eqnarray}\label{e3}
\mbox{\it $v_j\stackrel{i_0+2(x-1)+2j}{\to}r^+$ for every $j$ in $[y]$.}
\end{eqnarray}
In this case, the cat returns the statement
``${\rm dist}_G(r,m_i)\leq 4\Delta-6$'' in line \ref{l29},
where $i=i_0+2x+2y-2$,
and terminates the computation.
To show the correctness of {\sc Cat on a Tree},
we need to argue that this statement is correct.
Indeed,
if $m_i$ belongs to $T_{u_j}$ for some $j$ in $[x-1]$,
then, by (\ref{e2}),
$u_j\stackrel{i_0+2j}{\to}r\in I$,
which implies that 
\begin{eqnarray*}
{\rm dist}_T(r,m_i)
&\leq & i-(i_0+2j)\\
&=& (i_0+2x+2y-2)-(i_0+2j)\\
&\stackrel{j\geq 1}{\leq}& (i_0+2x+2y-2)-(i_0+2)\\
&=& 2x+2y-4\\
&\leq & 2\Delta+2(\Delta-1)-4\\
&=& 4\Delta-6,
\end{eqnarray*}
where we used $x\leq \Delta$ and $y\leq \Delta-1$.

If $m_i$ belongs to $T_{v_j}$ for some $j$ in $[y]$,
then, by (\ref{e3}),
$v_j\stackrel{i_0+2(x-1)+2j}{\to}r^+\in I$,
which implies that 
\begin{eqnarray*}
{\rm dist}_T(r,m_i)
&\leq & {\rm dist}_T(r,r^+)+i-(i_0+2x+2j-2)\\
&=& 1+(i_0+2x+2y-2)-(i_0+2x+2j-2)\\
&\stackrel{j\geq 1}{\leq}&  1+(i_0+2x+2y-2)-(i_0+2x)\\
&=& 2y-1\\
&\leq & 2\Delta-3\\
&\leq & 4\Delta-6,
\end{eqnarray*}
where we used $y\leq \Delta-1$.

If $m_i\in \{ r,r^+\}$, then 
${\rm dist}_T(r,m_i)\leq 1\leq 4\Delta-6$,
where we used $\Delta\geq 2$.

Finally, if $m_i$ does not belong to $T_r$,
then, by (\ref{e1}),
$r^-\stackrel{i_0}{\to}r\in I$,
which implies that 
\begin{eqnarray*}
{\rm dist}_T(r,m_i)
&\leq & i-i_0\\
&=& (i_0+2x+2y-2)-i_0\\
&=& 2x+2y-2\\
&\leq & 2(\Delta-1)+2(\Delta-1)-2\\
&=& 4\Delta-6,
\end{eqnarray*}
where we used $x\leq \Delta-1$ and $y\leq \Delta-1$.
Note that the existence of $r^-$ implies that $r$ has at most $\Delta-1$ children.
Altogether, the correctness of {\sc Cat on a Tree} follows,
and we consider its running time.

The outer {\bf while}-loop is executed at most 
$\max\{ 0,h-(4\Delta-6)\}$ times.
If, in some execution of the outer {\bf while}-loop, 
$r$ has $x$ children, and $r^+$ has $y$ children, 
then the first inner {\bf while}-loop is executed $x-1\leq \Delta-1$ times,
while the second inner {\bf while}-loop is executed at most $y\leq \Delta-1$ times.
Since in each of these executions of the inner {\bf while}-loops,
the discrete time proceeds exactly two units, 
{\sc Cat on a Tree} terminates in some round $i$
with $i\leq 2\max\{ 0,h-(4\Delta-6)\}((\Delta-1)+(\Delta-1))=O(h\Delta)$,
which completes the proof. 
\end{proof}
For $\Delta\geq 3$, with some more case analysis, it is possible to improve $4\Delta-6$ to $4\Delta-7$.
For $\Delta=2$, that is, on a path, say of order $n$, the cat can localize the mouse up to distance $2$ within time $O(\log n)$.
The corresponding strategy is similar to the strategy used in the following proof.

For integers $i$ and $j$,
let $[i,j]$ equal $\{ k\in \mathbb{Z}:i\leq k\leq j\}$.

\begin{proof}[Proof of Theorem \ref{theoremgrid}]
Let $[n]\times [n]$ denote the vertex set of $P_n\Box P_n$, where two vertices $(x,y)$ and $(x',y')$ are adjacent if and only if $|x-x'|+|y-y'|=1$. 

Suppose that 
\begin{eqnarray}\label{e4}
M_i\subseteq \Big[x_i,x_i+\partial x_i\Big]\times \Big[y_i,y_i+\partial y_i\Big]\subseteq \Big[n\Big]\times \Big[n\Big]
\end{eqnarray}
for some positive integers $i$, $x_i$, $\partial x_i$, $y_i$, and $\partial y_i$. For $i=1$, for instance, $M_1\subseteq [1,1+(n-1)]\times [1,1+(n-1)]=[n]\times [n]$. 
If $\partial x_i$ is even, then let $p_x=\frac{\partial x_i}{2}-1$, 
and
if $\partial x_i$ is odd, then let $p_x=\frac{\partial x_i}{2}-\frac{1}{2}$.
Let $p_y$ be defined similarly using $\partial y_i$.
Now, the cat chooses 
\begin{eqnarray*}
c_{i+1}&=&(x_i+p_x,y_i+p_y),\\ 
c_{i+2}&=&(x_i+p_x+2,y_i+p_y)\mbox{, and}\\
c_{i+3}&=&(x_i+p_x+2,y_i+p_y+2).
\end{eqnarray*}
Arguing similarly as for the key observation exploited in the proof of Theorem \ref{theoremtree}, it follows that 
\begin{eqnarray}\label{e5}
M_{i+3} &\subseteq &
\Big[x_{i+3},x_{i+3}+\partial x_{i+3}\Big]\times \Big[y_{i+3},y_{i+3}+\partial y_{i+3}\Big]\subseteq \Big[n\Big]\times \Big[n\Big],\mbox{ where}\label{e5}\\
\partial x_{i+3} & \leq & \left\lceil\frac{\partial x_i}{2}\right\rceil+4\label{e6},\\
\partial y_{i+3} & \leq & \left\lceil\frac{\partial y_i}{2}\right\rceil+3,\mbox{ and}\label{e7}
\end{eqnarray}
$x_{i+3}$,
$\partial x_{i+3}$,
$y_{i+3}$, and
$\partial y_{i+3}$ are positive integers.
Indeed, suppose that $\partial x_i$ is even, $\partial y_i$ is odd,
$d_{i+2}\leq d_{i+1}$, and $d_{i+3}>d_{i+2}$. In this case, it follows that 
\begin{eqnarray*}
m_i & \in & \Big[x_i,x_i+\partial x_i\Big]\times \Big[y_i,y_i+\partial y_i\Big],\\
m_{i+1} & \in & \Big[x_i-1,x_i+\partial x_i+1\Big]\times \Big[y_i-1,y_i+\partial y_i+1\Big],\\
m_{i+2} & \in & \Big[x_i+\frac{\partial x_i}{2},x_i+\partial x_i+2\Big]\times \Big[y_i-2,y_i+\partial y_i+2\Big],\mbox{ and}\\
m_{i+3} & \in & \Big[x_i+\frac{\partial x_i}{2}-1,x_i+\partial x_i+3\Big]\times \Big[y_i-3,y_i+\frac{\partial y_i}{2}+\frac{1}{2}\Big].
\end{eqnarray*}
See Figure \ref{fig1} for an illustration.
\begin{figure}[H]
\begin{center}
\unitlength 2mm 
\linethickness{0.4pt}
\ifx\plotpoint\undefined\newsavebox{\plotpoint}\fi 
\begin{picture}(24,18)(0,0)
\put(0,4){\framebox(20,14)[cc]{}}
\put(0,6){\line(1,0){20}}
\put(0,8){\line(1,0){20}}
\put(0,10){\line(1,0){20}}
\put(0,12){\line(1,0){20}}
\put(0,14){\line(1,0){20}}
\put(0,16){\line(1,0){20}}
\put(2,18){\line(0,-1){14}}
\put(4,18){\line(0,-1){14}}
\put(6,18){\line(0,-1){14}}
\put(8,18){\line(0,-1){14}}
\put(10,18){\line(0,-1){14}}
\put(12,18){\line(0,-1){14}}
\put(14,18){\line(0,-1){14}}
\put(16,18){\line(0,-1){14}}
\put(18,18){\line(0,-1){14}}
\put(20,18){\line(0,-1){14}}
\put(8,10){\circle*{0.5}}
\put(12,10){\circle*{0.5}}
\put(12,14){\circle*{0.5}}
\put(10,4){\line(0,-1){4}}
\put(10,2){\vector(1,0){3}}
\put(16,2){\makebox(0,0)[cc]{$m_{i+2}$}}
\put(20,12){\line(1,0){4}}
\put(22,12){\vector(0,-1){3}}
\put(23,7){\makebox(0,0)[cc]{$m_{i+3}$}}
\put(7,9){\makebox(0,0)[cc]{$1$}}
\put(13,9){\makebox(0,0)[cc]{$2$}}
\put(13,15){\makebox(0,0)[cc]{$3$}}
\end{picture}
\end{center}
\caption{An illustration for the case 
$\partial x_i=10$, $\partial y_i=7$, $d_{i+2}\leq d_{i+1}$, and $d_{i+3}>d_{i+2}$.}\label{fig1}
\end{figure}
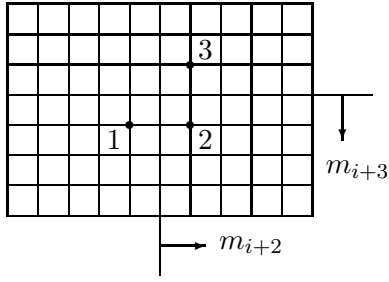
We obtain 
\begin{eqnarray*}
\partial x_{i+3}&=&\Big(x_i+\partial x_i+3\Big)-\left(x_i+\frac{\partial x_i}{2}-1\right)=\frac{\partial x_i}{2}+4\mbox{ and}\\
\partial y_{i+3}&=&\left(y_i+\frac{\partial y_i}{2}+\frac{1}{2}\right)-\Big(y_i-3\Big)=\left\lceil\frac{\partial y_i}{2}\right\rceil+3,
\end{eqnarray*}
and (\ref{e5}), (\ref{e6}), and (\ref{e7}) follow.
All remaining cases are similar; we leave the details to the reader.

Note that 
$\left\lceil\frac{x}{2}\right\rceil+4<x$ for $x\geq 10$,
and that 
$\left\lceil\frac{y}{2}\right\rceil+3<y$ for $y\geq 8$.
Therefore, iteratively choosing the vertices as above, the cat can ensure that (\ref{e4}) holds 
with $\partial x_i\leq 9$ and $\partial y_i\leq 7$ for some $i=O(\log n)$.
If either $\partial x_i<9$ or $\partial y_i<7$, then we obtain ${\rm rad}_G(M_i)\leq 8$.
If $\partial x_i=9$ and $\partial y_i=7$, 
then choosing $c_{i+1}$, $c_{i+2}$, and $c_{i+3}$ with the roles of the $x$- and $y$-coordinates exchanged,
we obtain,
using (\ref{e6}) and (\ref{e7}),
that (\ref{e5}) holds with 
$\partial x_{i+3}\leq \left\lceil\frac{9}{2}\right\rceil+3=8$ and 
$\partial y_{i+3}\leq \left\lceil\frac{7}{2}\right\rceil+4=8$,
which implies ${\rm rad}_G(M_{i+3})\leq 8$.
This completes the proof. 
\end{proof}

\begin{proof}[Proof of Theorem \ref{theoremslow}]
We show that
until the cat has not identified the position of the mouse in some round, 
for every positive integer $i$ with $i\equiv 1\mod 4\Delta$, 
the cat can choose a vertex $r_i$ 
such that it can 
\begin{itemize}
\item either identify the vertex $m_{i+4\Delta-1}$
in round $i+4\Delta-1$
\item or choose a vertex $r_{i+4\Delta}$ 
with ${\rm dist}_G(r_{i+4\Delta},m_{i+4\Delta})<{\rm dist}_G(r_i,m_i)$, 
\end{itemize}
which clearly implies the desired result.

Let $r_1$ be any vertex of $G$.
Now, suppose that the positive integer $i$ 
is such that $i\equiv 1 \mod 4\Delta$,
and that $r_i$ has been chosen.
Note that $m_i=m_{i+1}=\ldots=m_{i+4\Delta-1}$.
The cat chooses
$$\left(c_i,c_{i+1},\ldots,c_{i+2d}\right)
=\left(r_i,u_1,r_i,u_2,r_i,\ldots,r_i,u_d,r_i\right),$$
where $N_G(r_i)=\{ u_1,\ldots,u_d\}$ for some $d\leq \Delta$.

If 
${\rm dist}_G(u_j,m_i)={\rm dist}_G(c_{i+2j-1},m_i)
>{\rm dist}_G(c_{i+2j-2},m_i)={\rm dist}_G(r_i,m_i)$
for every $j\in [d]$,
then $m_{i+4\Delta-1}=m_i=r_i$, 
that is, the cat identifies $m_{i+4\Delta-1}$ in round $i+4\Delta-1$.
Hence, we may assume that 
${\rm dist}_G(r_i,m_i)=
{\rm dist}_G(c_{i+2j},m_i)
>{\rm dist}_G(c_{i+2j-1},m_i)=
{\rm dist}_G(u_j,m_i)$ for some $j\in [d]$.
In this case, the cat chooses
$$\left(c_{i+2d+1},c_{i+2d+2},\ldots,c_{i+2d+2d'+1}\right)
=\left(u_j,v_1,u_j,v_2,u_j,\ldots,u_j,v_{d'},u_j\right),$$
where $N_G(u_j)=\{ r_i,v_1,\ldots,v_{d'}\}$
for some $d'\leq \Delta-1$.

If 
${\rm dist}_G(v_\ell,m_i)={\rm dist}_G(c_{i+2d+2\ell},m_i)
>{\rm dist}_G(c_{i+2d+2\ell-1},m_i)={\rm dist}_G(u_j,m_i)$
for every $\ell\in [d']$,
then $m_{i+4\Delta-1}=m_i=u_j$.
Hence, we may assume that 
${\rm dist}_G(u_j,m_i)=
{\rm dist}_G(c_{i+2d+2\ell+1},m_i)
>
{\rm dist}_G(c_{i+2d+2\ell},m_i)
={\rm dist}_G(v_\ell,m_i)$
for some $\ell\in [d]$,
which implies
${\rm dist}_G(v_\ell,m_i)={\rm dist}_G(u_j,m_i)-1={\rm dist}_G(r_i,m_i)-2$.
In this case, the cat chooses $r_{i+4\Delta}$ equal to $v_\ell$.
Since $m_{i+4\Delta}\in N_G[m_{i+4\Delta-1}]=N_G[m_i]$,
we obtain 
${\rm dist}_G(r_{i+4\Delta},m_{i+4\Delta})
\leq{\rm dist}_G(r_{i+4\Delta},m_{i+4\Delta-1})+1
={\rm dist}_G(v_\ell,m_i)+1
={\rm dist}_G(r_i,m_i)-1
$,
which completes the proof.
\end{proof}

\section{Conclusion}\label{sec3}

We collect some problems for further research.
It would be interesting to obtain tight bounds for the problems considered in Theorem \ref{theoremtree} and Theorem \ref{theoremgrid}.
In this context, lower bounds on the minimum distance up to which the mouse can be located would be useful.
An interesting special graph, and possible counterexample to Conjecture \ref{conjecture}, seems to be the tree 
that arises by subdividing each edge of the star $K_{1,k}$
exactly $k-1$ times.
One may consider different slowness conditions such as, 
for instance,
``$(m_i\not=m_{i-1})\wedge (m_j\not=m_{j-1})\wedge(j>i)\Rightarrow j-i\geq k$'' for some integer $k$ corresponding to the inverse maximum speed of the mouse.
For a proof of Conjecture \ref{conjecture},
it might be useful that, 
for every graph $G$ of maximum degree at most $\Delta$, 
and every set $M$ of vertices of $G$,
there is an edge $uv$ of $G$ 
for which the two sets
$\{ m\in M:{\rm dist}_G(u,m)>{\rm dist}_G(v,m)\}$
and
$\{ m\in M:{\rm dist}_G(u,m)<{\rm dist}_G(v,m)\}$
both contain at least $\frac{|M|-1}{\Delta}$ vertices.

\end{document}